\documentclass{article}
\usepackage[utf8]{inputenc}
\usepackage[english]{babel}
\usepackage{polski}
\usepackage[colorlinks=true, citecolor=blue, linkcolor=blue, linktocpage=true]{hyperref}

\newcommand{\FOLS}{\mathsf{FOL\BOX}}
\newcommand{\A}{\mathcal{A}}
\newcommand{\B}{\mathcal{B}}
\newcommand{\C}{\mathfrak{C}}
\newcommand{\U}{\mathfrak{U}}
\newcommand{\D}{\mathcal{D}}
\newcommand{\W}{\mathcal{W}}

\newcommand{\LFB}{\mathcal{L}(\forall, \BOX)}
\usepackage{amsmath}
\usepackage{amsthm}
\usepackage{amsfonts}
\usepackage{amssymb}
\usepackage{amsbsy}
\usepackage{latexsym}

\usepackage{amsmath}

\usepackage{breqn}

\newtheorem{theorem}{Theorem}
\newtheorem{lemma}{Lemma}

\usepackage{amsmath}

\usepackage{breqn}

\usepackage{stackengine}
\newcommand\BOX{\stackMath\mathbin{\stackinset{c}{0.1ex}{a}{0.4ex}{\mbox{\tiny{$\pmb{\vdash}$}}}{\Box}}}
\newcommand\POS{\stackMath\mathbin{\stackinset{c}{0.2ex}{a}{0.35ex}{\mbox{\tiny{$\pmb{\vdash}$}}}{\Diamond}}}
\title{First-order Logic with\\
 Being a Thesis Modal Operator} 
\author{Marcin Łyczak}
\date{
}
\begin{document}
\maketitle
\begin{abstract}
We introduce syntactic modal operator $\BOX$ for \textit{being a thesis} into first-order logic. This logic is a modern realization of R. Carnap's old ideas on modality, as logical necessity (J. Symb. Logic, 1946) \cite{Ca46}. We place it within the modern framework of quantified modal logic with variable domains. We prove the completeness theorem using a property of a kind of normal form, and the fact that in the canonical frame, the relation on all maximal consistent sets, $R = \{\langle \Gamma, \Delta \rangle\colon \forall A (\BOX A \in \Gamma \Longrightarrow A \in \Delta)\}$, is a universal relation (both are unprovable in known quantified modal logics). The strength of the $\BOX$ operator is a proper extension of modal logic $\mathsf{S5}$. Using completeness, we prove that satisfiability in a model of $\BOX A$ under arbitrary valuation implies that $A$ is a thesis of formulated logic. So we can syntactically define logical entailment and consistency. Our semantics differ from S. Kripke's standard one \cite{Kr63} in syntax, semantics, and interpretation of the necessity operator. We also have free variables, contrary to Kripke's and Carnap's approaches, but our notion of substitution is non-standard (variables inside modalities are not free). All $\BOX$-free first-order theses are provable, as well as the Barcan formula and its converse. Our specific theses are \linebreak[4] $\BOX A \to \forall x A$, $\neg \BOX (x = y)$, $\neg \BOX \neg (x = y)$, $\neg \BOX P(x)$, $\neg \BOX \neg P(x)$. We also have $\POS \exists x A(x) \to \POS A(^{y}/_{x})$, and $\forall x \BOX A(x) \to \BOX A(^{y}/_{x})$, if $A$ is a $\BOX$-free formula.
\end{abstract}
\section{Introduction}
In 1946, R. Carnap formulated the \textit{modal functional calculus} ($\mathsf{MFC}$), which is a version of quantified modal logic with an operator intended to express \textit{logical necessity} \cite{Ca46}. Carnap's logic is not a typical modal quantified logic. It does not contain free variables, is stronger than $\mathsf{S5}$ (formulas like $\neg \Box \exists x Fx$ are provable), and identity holds only between the same variables or constants, although identical objects are necessarily identical. In short, $\mathsf{MFC}$ is far from contemporary first-order modal logic. Its models were built from sets of atomic formulas and their negations. There is no valuation and interpretation of predicates. Differences and similarities to Kripke's semantics can be found in \cite{Le20}. Recently, two important results on Carnap's logic have been published. M. Cresswell proved the completeness of a variant of Carnap's axiomatization, keeping Carnap's original semantics almost untouched \cite{Cr14}. M. Madows formulated modern non-modal semantics for Carnap's logic, together with a corresponding tableau system \cite{Ma12}. Madows interpreted logical necessity in a first-order structures as a quantification over all (parametrically definable) subsets of the domain of the structure. We believe that his approach can be considered a faithful modern representation of Carnap's approach to logical necessity. Carnap's necessity, as well as the necessities from the two cited works, actually does not express being a first-order thesis. For example, formulas like $\Box A \to \forall x A$, $\neg \Box \exists xy (x \not= y)$, and $\neg \Box \exists xy (x = y)$ should be provable if we were to think of $\Box$ as a logical necessity.\smallskip\\
As far as we know, there is no formalized version of Carnap's ideas in the literature where logical necessity explicitly expresses being a first-order thesis. We aim to build a first-order modal logic with an operator expressing first-order theses and interpret it in a variant of variable domain semantics. For simplicity, we do not introduce function symbols and constants. To show that our modality $\BOX$ is factually correct, we prove that the satisfiability of any formula $\BOX A$ in any model implies its logical validity and being a thesis, as we prove the completeness theorem.
\section{Language and its semantics}
Let $\mathcal{L}(\forall, \BOX)$ be a first-order language with identity and the modal operator $\BOX$, built with a set of predicates $Pred=\{P^{i}_{j}\}_{i\in I,j\in J}$, and a set of variables $Var$, with $x$ ranging over them.
The language $\LFB$ allows arbitrary iterations and is an ordinary language of quantified modal logic. The difference is that the operator $\BOX$ is read as \textit{is a thesis} (or eventually \textit{is logically necessary}). Other logical symbols are defined in the usual way, and the dual operator $\POS =: \neg \!\BOX\! \neg$ is read as \textit{ is consistent} (or eventually \textit{is logically possible}).\smallskip\\
We say that a formula is $\BOX$-\textit{free} when it does not contain any modalities. We denote the set of the $\BOX$-\textit{free} first order theses as $\mathsf{FOL}$. \smallskip\\
Let $\C$ be the class of all first-order structures with variable domains, i.e.,
$$\A=\langle \W, \{\D_{w}\}_{w\in \W}, I \rangle,$$ where
$\W$ is a non-empty domain of \textit{possible worlds}; $\D_{w}$ is the domain of \textit{objects} of the world $w$; and $I$ is an \textit{interpretation} function that for any $P^{n}_{j}\!\in\! Pred$ and $w\!\in\! \W$, assigns $n$-ary pairs of objects of the world, i.e., $I(P^{n}_{j},w)\!\subseteq\! \D^{n}_{w}$.\smallskip\\
\textit{Valuation} in a given structure is world-dependent and is any function $v$ that assigns, for any variable and possible world of the structure, an object from the domain of the world, i.e., $v(x,w)\in \D_{w}$. \smallskip\\
A \textit{variant of valuation} $v$ \textit{in a world} $w$, which maps $x$ to $a\in \D_{w}$ and has all other values like $v$ on all other variables and worlds except possibly at the variable $x$, is denoted as $v^{a}_{x}(w)$.\smallskip\\
For any structure $\A=\langle \mathcal{W}, \{\D_{w}\}_{w\in \W}, I \rangle$, belonging to $\C$, world $w\in \W$, and valuation $v$ on $\A$, we define satisfaction:
\begin{align*}
\A, v, w \models P^{n}_{j}(x_{1},...,x_{n}) &\Longleftrightarrow \langle v(x_{1},w),...,v(x_{n},w)\rangle\in I({P}^{n}_{j},w);\\
\A, v, w   \models x=y &\Longleftrightarrow  v(x,w)=v(y,w);\\
\A, v, w  \models \neg A &\Longleftrightarrow  \A, v, w\not \models A;\\
\A, v, w  \models A \to B &\Longleftrightarrow  \A, v, w \not \models A \mbox{ or } \A, v, w \models B;\\
\A, v, w  \models \forall x A &\Longleftrightarrow  \A,  v^{a}_{x}(w), w \models A ,  \mbox{ for every } a\in D_{w};
\\
\qquad \qquad \A, v, w  \models \BOX A &\Longleftrightarrow  \A, v', w' \models A,  \mbox{ for every } v' \mbox{ in } \A \mbox{ and } w'\in \W.
\end{align*}
We write $\A \models A$ when $\A, v, w \models A$ for every $v,w$ in $\A$.\smallskip\\
For now, the $\BOX$ operator does not yet express being a first-order thesis. To change this, we have to narrow the class $\C$ to its proper subclass $\U\C \subset \C$, which fulfills the following condition:
\begin{equation}
\tag{$\U\C$}\label{!}
\A\in \U\C \Longleftrightarrow (\A \models A \Longrightarrow \forall_{\B\in \C} \B\models A), \mbox{ for any } \BOX\mbox{-free formula } A.
\end{equation}
The class $\U\C \subseteq \C$ is called \textit{universal}, and its members are called \textit{universal structures}. Let us comment on the class $\U\C$ in which we will interpret the logic we aim to build:
\begin{itemize}\setlength\itemsep{-0.3em}
\item Any universal structure, i.e., fulfilling condition \eqref{!}, can falsify any $\BOX$-free formula that is not a first-order thesis.
\item Any universal structure has an infinite domain of possible worlds $\W$ and an infinite domain of objects $\bigcup \{D_{w}\}_{w\in \W}$.
\item There are universal structures with countable domains of worlds and individuals.
\item Two universal structures may be non-isomorphic and have different properties.
\end{itemize}
For any formula $A \in \LFB$, we say that:
\begin{align}
 \notag A \mbox{ is \textit{logically valid} } \mbox{ \ iff \ }&  \A \!\models\! A \mbox{ for any } \A \in \U\C.
\end{align}
\section{Formal system for being a thesis operator} \label{formal system}
It is a well-known fact that every proper extension of modal logic
$\mathsf{S5}$ is either not closed under the substitution rule or is closed under substitution
but is characterized by a single frame with a finite domain, i.e., a frame with a fixed number of possible worlds \cite{Sc51}. Given that our models are infinite (for worlds and objects), we will formulate slightly different notions of free variable and substitution than the standard one. Our new notions and their counterparts from first-order logic are equivalent for formulas without $\BOX$, but they differ for formulas with $\BOX$ to the one known from quantified modal logic.\newpage
The system $\FOLS$, for first-order logic with a thesis operator, is defined as the smallest set that contains all instantiations of all tautologies of classical logic, and the following schemes of axioms:
\begin{gather}
\tag{$\mathtt{K}$} \BOX (A \to B)\to (\BOX A\to \BOX B),\\
\tag{$\mathtt{T}$} \BOX A \to A,\\
\tag{$\mathtt{5}$} \neg \!\BOX\! A \to \BOX \neg \BOX A,\\
\tag{$\forall 1$} \label{forall1}\forall \pmb{x} A \to A(^{\pmb{y}}/_{\pmb{x}}),
\mbox{ where } A(^{\pmb{y}}/_{\pmb{x}}) \mbox{ is }\forall\!\BOX\!\mbox{\textit{-free substitution}},\\
\tag{$\forall 2$} \label{forall2}\forall \pmb{x} (A \to B)\to  (A \to \forall \pmb{x} B), \mbox{ if } \pmb{x} \mbox{ is not a } \forall\!\BOX\!\mbox{\textit{-free} variable in }A ,\\
\tag{ID} x=x,\\
\tag{=}  x=y\land A(x)\to A(^{y}//_{x}) , \mbox{ if } A \mbox{ is a }\BOX\mbox{-free formula},\\
\tag{$\mathtt{mix}$} \label{mix}\BOX A \to \forall \pmb{x} A,\\
\tag{$\mathtt{!}$}\label{6} \neg \!\BOX\! A, \mbox{ if } A \notin \mathsf{FOL}.
\end{gather}
The set $\FOLS$ is closed under the following rules:
\begin{gather}
\tag{$\mathtt{RG}$}\label{genBOX} A\in \FOLS \Longrightarrow \BOX A \in \FOLS,\\
\tag{$\mathtt{MP}$}A\to B, A\in \FOLS\Longrightarrow B\in  \FOLS.
\end{gather}
The \eqref{mix} schema states that logical theses concern all objects. The generalization rule for universal quantifier $A\in \FOLS\Longrightarrow \forall x A\in \FOLS$ is derivable from \eqref{genBOX} and \eqref{mix}. \eqref{6} with \eqref{genBOX} guarantee that $\BOX A \in \FOLS$ or $\neg \BOX A \in \FOLS$, for any modality-free formula $A$, but as we show later it is a case of any formula $A$.\smallskip\\
An occurrence of a variable $x$ in a formula $A$ is $\forall\BOX$\textit{-bound} if $x$ is in the scope
of quantifier or $x$ occur in the formula $A$ in range of $\BOX$ operator. Otherwise, the occurrence is $\forall\BOX$\textit{-free}.\smallskip\\
A variable $x$ is $\forall\BOX$\!\textit{-free} in the formula $A$ if it has at least one $\forall\BOX$\!-free occurrence.\smallskip\\
A $\forall\BOX$-free substitution $A(^{y}/_{x})$ is the simultaneous replacement of every $\forall\BOX$-free occurrence of $x$ by $y$ in $A$, provided that $y$ is substitable for a $x$ in $A$; where variable $y$ \textit{is substitable for} $\forall\BOX$-free occurrence of $x$ in a formula $A$, if $y$ is not captured by a quantifier when it replaces $x$.\smallskip\\
Let us note that:
\begin{itemize}\setlength\itemsep{-0.2em}
\item In $\neg \!\BOX \! Px\land \exists x (Px)$ both occurrences of $x$ are $\forall\BOX$-bound.
\item  In a formula $Px \land \POS \neg Px$ the first occurrence of a variable $x$ is $\forall\BOX$-free while the second occurrence is  $y$ is
$\forall\BOX$-bound.
\item  If $A=Px \land \POS (Px \land \neg Py)$ then correct  $\forall\BOX$-substitution are: \linebreak[4
] $A(^{y}/_{x})=P y \land \POS (Px \land \neg Py)$ and $A(^{z}/_{x})= P z \land \POS (Px \land \neg Py)$.
\item  If $A=Px \land \POS (Px \land \neg Py))$ then the following is not correct $\forall\BOX$-substitution:  $A(^{x}/_{y}) = P y \land \POS (Py \land \neg Py)$.
\end{itemize}

\newpage
\section{Soundness and completeness}
We start with the satisfaction theorem.
\begin{theorem}\label{satisfation}
If $A\in \FOLS$, then $A$ is logically valid.
\end{theorem}
\begin{proof} 
The logical validity of \eqref{mix} follows directly from the satisfaction conditions for $\BOX$ and $\forall$. We show only the proof of the logical validity of the axioms from schema \eqref{6}, since the rest of the cases are routine. We assume that $\neg \BOX A$ is not logically valid for some $A \notin \mathsf{FOL}$. Thus, we have $\A, v, w \models \BOX A$ for some $\A\in \U\C$, $v$, and $w$ from the domain of $\A$. Now, by satisfaction we have $\A, v, w \models A$ for every $v, w$ in $\A$, and so $\A \models A$. Thus, since $A$ is $\BOX$-free, condition \eqref{!} guarantee that in any structure $\B\in \C$ we have $\B\models A$.
Given $A \notin \mathsf{FOL}$ and the completeness of $\mathsf{FOL}$, we know that there exists an ordinary first-order structure $\mathcal{F}$ and a valuation $v$ such that $A$ is not satisfied in $\mathcal{F}$ under $v$. Thus, there is a structure $\mathcal{B} \in \C$ (e.g., analogous to $\mathcal{F}$ and $v$ with one possible world $w$) such that $\mathcal{B}, v, w \not\models A$, which gives a contradiction.
\end{proof}
Now we prove the completeness of $\FOLS$. \medskip\\
The definitions of ($\FOLS$-)\textit{consistency} and ($\FOLS$-)\textit{maximal consistency} in $\FOLS$ are standard. A set $\Gamma$ of formulas is $\exists$-\textit{saturated} if for any formula $\neg \forall \pmb{x} A \in \Gamma$, there exists a variable $\pmb{y}$ and a $\forall\BOX$-free substitution such that $\neg A (^{\pmb{y}}/_{\pmb{x}}) \in \Gamma$.\smallskip\\
In a standard way, we can prove the following lemma
\begin{lemma}\label{lemmaconsistentset1}
For any $\FOLS$-consistent set $\Sigma$, there exists a set $\Gamma$ such that $\Sigma \subseteq \Gamma$ and $\Gamma$ is $\exists$-\textit{saturated} maximal $\FOLS$-consistent with the usual properties$:$
\begin{align}
\tag{$\Gamma_{\vdash}$}\label{gammaAM} \phantom \Gamma& \FOLS\subseteq \Gamma, \\
\tag{$\Gamma_{\neg}$}\label{gammaneg} A \in \Gamma &\Longleftrightarrow \neg A \not \in \Gamma, \\
\tag{$\Gamma_{\to}$}\label{gammato} A\to B \in \Gamma &\Longleftrightarrow A \not \in \Gamma \mbox{ or } B \in \Gamma.
\end{align}
\end{lemma}
Since $\BOX$ is a normal modality we also know that
\begin{lemma}\label{lemmaconsistentset2}
If $\Gamma$ is a $\FOLS$-consistent set and $\POS A \in \Gamma$, then $\{B \colon \BOX B \in \Gamma\} \cup \{A\}$ is also a $\FOLS$-consistent set.
\end{lemma}
Now we introduce a kind of normal form of
formulas:
\begin{equation*}
\begin{aligned}
D \text{ is }\textit{an elementary disjunction } \text{ iff }& D :=A\lor \POS B \lor  \BOX C_{1}\lor \dots \lor \BOX C_{n},\\
& \text{  for some } \BOX\mbox{-free formulas } A, B, C_{1},\dots C_{n} ;\smallskip\\
F\mbox{ is }\textit{a conjunctive form} \mbox{ iff } & F \mbox{
is a conjunction of elementary disjunctions}.
\end{aligned}
\end{equation*}
Let $\mathcal{MC}$ be the family of all $\exists$-\textit{saturated} maximal $\FOLS$-consistent sets.
\begin{lemma}\label{lemmaBOXinWc}
For any maximal $\FOLS$-consistent sets $\Gamma, \Delta\in \mathcal{MC}$, $\BOX$-free formula $A$, and any elementary disjunction $D\colon$
\begin{gather}
\tag{$i$}\label{i} \forall_{\Gamma,\Delta\in \mathcal{MC}}(\BOX A\in \Gamma \Longleftrightarrow \BOX A\in \Delta);\\
\tag{$ii$}\label{ii} \forall_{\Gamma,\Delta\in \mathcal{MC}}(\POS A\in \Gamma \Longleftrightarrow \POS A\in \Delta);\\
\tag{$iii$}\label{iii} \forall_{\Gamma,\Delta\in \mathcal{MC}}(\BOX D\in \Gamma \Longleftrightarrow \BOX D\in \Delta).
\end{gather}
\end{lemma}
\begin{proof}
We start with \eqref{i} and for any $\BOX$-free formula $A$ and maximal \linebreak[4] $\FOLS$-consistent sets $\Gamma,\Delta$ we assume indirectly that $\BOX A\in \Gamma,$ and $\BOX A\notin \Delta$. Since $\Delta$ is a maximal consistent set we have $\POS \neg A\in \Delta$ by \eqref{gammaneg}. The latter using Lemmas \ref{lemmaconsistentset2} and \ref{lemmaconsistentset1} gives that there is a maximal consistent set $\Delta_{0}\in \mathcal{MC}$ such that $\neg A \in \Delta_{0}$, and next $A \notin \Delta_{0}$. From \eqref{gammaAM} we know that $\FOLS\subseteq \Delta_{0}$, so $A\notin \FOLS$, since $A \notin \Delta_{0}$. The latter applying axiom \eqref{6} gives $\neg \BOX A\in \FOLS$, since $A$ is $\BOX$-free. Having $\neg \!\BOX \! A\in \FOLS$ we apply again \eqref{gammaAM} but this time to $\Gamma$ and obtain $\neg \BOX A\in \Gamma$  which gives a contradiction with the first assumption. \smallskip\\
For \eqref{ii} we just use \eqref{i} with $A/\neg A$ and apply \ref{gammaAM} of Lemma \ref{lemmaconsistentset1}.\smallskip\\
For \eqref{iii} we assume indirectly that  $\BOX D\in \Gamma$, and $\BOX D\notin \Delta$, for some $\Gamma,\Delta\in \mathcal{MC}$.  $\BOX D\notin \Delta$ gives $\POS \neg D \in \Delta$ by \eqref{gammaneg} so applying Lemma \ref{lemmaconsistentset2} there is a maximal consistent set $\Delta_{0}\in \mathcal{MC}$ such that $D\notin \Delta_{0}$. Let $D:=A \lor \POS B \lor \BOX C_{1}\lor \dots \lor\BOX C_{n}$, since $D$ is an elementary disjunction. We have $A\notin \Delta_{0}$ and $\neg \POS \!B, \neg \BOX C_{1}, \dots, \neg \BOX  C_{n}\in \Delta_{0}$, by maximal consistency of $\Delta_{0}$ and classical logic. This using \eqref{i} and \eqref{ii} of Lemma \ref{lemmaBOXinWc} $\neg \POS \!B, \neg \BOX C_{1}, \dots, \neg \BOX  C_{n}\in \Delta_{0}$ implies \\
\indent ($\star$): $\forall_{\zeta \in \mathcal{MC}}(\neg \POS \!B, \neg \BOX C_{1}, \dots, \neg \BOX  C_{n}\in \zeta)$, \\
since $B, C_{j\ge n}$ is $\BOX$-free formula. The first assumption  $\BOX D\in \Gamma$, using \linebreak[4] $D:=A \lor \POS B \lor \BOX C_{1}\lor \dots \lor\BOX C_{n}$, gives us $\BOX (A \lor \POS B \lor \BOX C_{1}\lor \dots \lor\BOX C_{n} )\in \Gamma$. The latter entails $\BOX A \lor \POS B \lor \BOX C_{1}\lor \dots \lor\BOX C_{n} \in \Gamma$ because a $\mathsf{S5}$ property of $\BOX$ ($\BOX (A \lor \BOX B) \leftrightarrow (\BOX A \lor \BOX B)$) and maximal consistency of $\Gamma$. Thus using ($\star$) we obtain $\BOX A\in \Gamma$. Having $\BOX A\in \Gamma$ we obtain $\BOX A\in \Delta_{0}$ with the use of \eqref{i} from Lemma \ref{lemmaBOXinWc}, because $A$ is $\BOX$-free formula. However, $\BOX A\in \Delta_{0}$ with earlier obtained $A\notin \Delta_{0}$ implies that $\Delta_{0}$ is not consistent set.
\end{proof}
To prove that every $\LFB$ formula has equivalent conjunctive form we use the fact that $\BOX$ has properties of modal logic $\mathsf{S5}$, and some specific $\FOLS$ theses.
\begin{lemma} The following formulas are provable in $\FOLS$:
\begin{gather}
\tag{$\mathtt{1a}$}\label{1a} \BOX (A \lor \BOX B) \leftrightarrow (\BOX A \lor \BOX B);\\
\tag{$\mathtt{1b}$} \label{1b}\exists x \! \BOX \! A \leftrightarrow \forall x \! \BOX \! A  \leftrightarrow \BOX A;\\
\tag{$\mathtt{1c}$}  \exists x \! \POS \! A   \leftrightarrow\forall x \! \POS \! A  \leftrightarrow \POS A;\\
\tag{$\mathtt{1d}$}\label{1d} \forall x (A \lor \BOX B) \leftrightarrow (\forall x A \lor \BOX B).
\end{gather}
\end{lemma}
The first one, as we mentioned in the Proof of Lemma \ref{lemmaBOXinWc}, is provable in modal logic $\mathsf{S5}$, thus also in $\FOLS$. The next three are not probable in known quantifier modal logics. We omit elementary derivations of \eqref{1b}-\eqref{1d} in $\FOLS$ and use them to show that the logic $\FOLS$ has a property of equivalent conjunctive form.
\begin{lemma}\label{form}
For any $A\in \LFB$ there is a conjunctive form equivalent to $A$.
\end{lemma}
\begin{proof}
Proof by induction on $A$. The boolean steps are standard. Let $A=\BOX B$
and $B'= B_{1} \land \ldots \land B_{n}$ be conjunctive form of $B$. $\BOX$ is a normal modality thus $A\leftrightarrow \BOX B_{1} \land \ldots \land \BOX  B_{n}$ and \eqref{1a} guarantees that every $\BOX B_{i}$ have equivalent elementary disjunction, thus $A$ has an equivalent conjunctive form.
Let $A=\forall x B$
and $B'= B_{1} \land \ldots \land  B_{n}$ be a conjunctive form of $B$. We have \linebreak[4] $A\leftrightarrow \forall x B_{1} \land \ldots \land \forall x B_{n}$ and
\eqref{1d} guarantees that every $\forall x B_{i}$ have equivalent elementary disjunction, thus $A$ has an equivalent conjunctive form.
\end{proof}
\noindent Let  $R =\{\langle\Gamma, \Delta\rangle\colon
\forall_{A\in \mathcal{L}(\forall,\BOX)} ( \BOX A\!\in \!\Gamma \! \Longrightarrow \! A \!\in\! \Delta)\}$ be accessibility relation in $\mathcal{MC}$.
\begin{theorem}\label{universalR}
$(\mathcal{MC}, R)$ is a modal frame with a universal relation.
\end{theorem}
\begin{proof}We assume indirectly that $R$ is not universal. Let $\BOX A \in \Gamma$ and $A \notin \Delta$ for some $\Gamma, \Delta \in \mathcal{MC}$. According to Lemma \ref{form}, we know that $A$ has an equivalent conjunctive form consisting of elementary disjunctions, denoted as $D_{1} \land \dots \land D_{n}$. Since $\BOX$ is a normal modality, $\BOX A \in \Gamma$ gives $\BOX D_{1} \land \dots \land \BOX D_{n} \in \Gamma$. From the second assumption, we have $\neg D_{1} \lor \dots \lor \neg D_{n} \in \Delta$. Therefore, by Lemma \ref{lemmaconsistentset1}, there is an elementary disjunction $D_{i}$ such that $D_{i} \notin \Delta$. We already have $\BOX D_{i} \in \Gamma$. Using \eqref{iii} of Lemma \ref{lemmaBOXinWc}, we conclude that $\forall_{\zeta \in \mathcal{MC}} \BOX D_{i} \in \zeta$, since $D_{i}$ is an elementary disjunction. Thus, $\BOX D_{i} \in \Delta$ and consequently $D_{i} \in \Delta$, which means that $\Delta$ is not consistent, since $D_{i}\notin \Delta$.
\end{proof}
We did not find quantifier modal logics where an analog of the above lemma would hold. \smallskip\\ 
As an immediate consequence of the above lemma, we obtain that
\begin{lemma}\label{lemmamaxBOX2}
For any $A\in \mathcal{L}(\forall, \BOX)$, and $\Gamma,\Delta\in \mathcal{MC}$: $\BOX A\in \Gamma \Longleftrightarrow \BOX A\in\Delta.$
\end{lemma}
\noindent Now, for any $\Gamma \in \mathcal{MC}$, we define the relation $\sim_{\Gamma} = \{\langle x,y \rangle \colon x=y \in \Gamma\}$, which is equivalence relation by axioms for identity and construct the canonical structure $\A^{\mathcal{MC}} = \langle \mathcal{MC}, \{\D_{\Gamma}\}_{\Gamma\in \mathcal{MC}}, I^{\mathcal{MC}}\rangle$ as usual, i.e., $\D_{\Gamma}=\{[x]_{\sim_{\Gamma}}\colon x\in Var\}$, and next $I^{\mathcal{MC}}$, and $v^{\mathcal{MC}}$ on $\A^{\mathcal{MC}}$ are defined respectively. So canonical valuation $v^{\mathcal{MC}}$ in any possible world do not depend on choice of \textit{representatives} for the equivalence classes. \smallskip\\
Now, in non trivial case $A=\BOX B$ we use axiom  \eqref{mix}, and inductively prove 
\begin{lemma} $\A^{\mathcal{MC}}, v^{\mathcal{MC}},\Gamma \models A \Longleftrightarrow  A\in \Gamma$, for any $A\in \LFB $ and $\Gamma\in \mathcal{MC}$.
\end{lemma}
\noindent Using proved lemmas its easy to prove that
\begin{theorem}\label{completness}
If $A$ is logically valid, then $A \in \FOLS$.
\end{theorem}
\section{Conclusion}
In the end, we show that the satisfability of $\BOX A$ in any model implies that $A\in \FOLS$. The final theorem demonstrates that identifying the modality $\BOX$ with the operator of being a thesis (and only a thesis) is substantively correct.
\begin{theorem}
If there is a model on universal structure $\A\in \U\C$ such that $\A,w,v \models \BOX A$ then $A\in \FOLS$, for any formula $A$.
\end{theorem}
\begin{proof}
We assume indirectly that there is a formula $A$, a structure $\mathcal{A} \in \U\C$, a valuation $v$, and a world $w$ from the domain of possible worlds of $\mathcal{A}$ such that $\mathcal{A}, v, w \models \BOX A$, and that $A \notin \FOLS$. Using the second assumption and the completeness of $\FOLS$ from Theorem \ref{completness} we have that there is a maximal consistent set $\Delta \in \mathcal{MC}$ such that $A \notin \Delta$. We define a set $\Gamma = \{B \in \LFB \colon \mathcal{A}, v, w \models B\}$ that is maximal consistent, i.e., $\Gamma \in \mathcal{MC}$. The defined maximal consistent set, with the first assumption, gives $\BOX A \in \Gamma$. However, $\BOX A \in \Gamma$ and $A \notin \Delta$ cannot both be true because of Theorem \ref{universalR} on Krpike frame from $\mathcal{MC}$. 
\end{proof}

\end{document}